\documentclass[12pt,reqno]{amsart}

\usepackage[foot]{amsaddr}
\usepackage{fullpage, graphicx, amsmath, amssymb}
\usepackage{blindtext}

\theoremstyle{plain}
 \newtheorem{thm}{Theorem}[section]
 \newtheorem{prop}[thm]{Proposition}
 \newtheorem{lem}[thm]{Lemma}
 \newtheorem{cor}[thm]{Corollary}

\theoremstyle{definition}

 \newtheorem{exm}{Example}[section]
 \newtheorem{dfn}{Definition}[section]
  \newtheorem{rem}{Remark}[section]
 \newtheorem{nota}{Notation}[section]

 \numberwithin{equation}{section}
 
 \usepackage[pagebackref=false]{hyperref} % with basic options
\hypersetup{
    plainpages=false,       % needed if Roman numbers in frontpages
    unicode=false,          % non-Latin characters in Acrobat’s bookmarks
    pdftoolbar=true,        % show Acrobat’s toolbar?
    pdfmenubar=true,        % show Acrobat’s menu?
    pdffitwindow=false,     % window fit to page when opened
    pdfstartview={FitH},    % fits the width of the page to the window
    pdfnewwindow=true,      % links in new window
    colorlinks=true,        % false: boxed links; true: colored links
    linkcolor=blue,         % color of internal links
    citecolor=green,        % color of links to bibliography
    filecolor=magenta,      % color of file links
    urlcolor=cyan           % color of external links
}

\title{Connected Chord Diagrams and the Combinatorics of Asymptotic Expansions}
\author[1]{Ali Assem Mahmoud$^\ast$\;\;}
\address{$\ast$ Department of Mathematics, Faculty of Science, Cairo University, Egypt; and Department of Combinatorics and Optimization, University of Waterloo, ON, Canada.}
\email{ali.mahmoud@uwaterloo.ca}

\author[2]{\;\;Karen Yeats$^\dagger$}
\address{$\dagger$ Department of Combinatorics and Optimization, University of Waterloo, ON, Canada}
\email{kayeats@uwaterloo.ca}

\thanks{KY and AM would like to thank Gerald Dunne, Michi Borinsky, and Dirk Kreimer for useful discussions.  KY is supported by an NSERC Discovery grant and by the Canada Research Chair program.  During some of this work KY was supported by a Humboldt fellowship and thanks Dirk Kreimer for hosting her Humboldt visit to Berlin. AM would like to thank KY for initiating the research in this problem and for her continuous support.}

\date{}

\begin{document}

\maketitle

\begin{abstract}
   In this article we study an asymptotic expansion for $C_n$, the number of connected chord diagrams on $n$ chords. The expansion is obtained in earlier work by means of alien derivatives applied to the generating series of connected chord diagrams; we seek a combinatorial interpretation. The main outcome presented here is a new combinatorial interpretation for entry \href{https://oeis.org/A088221}{A088221} of the OEIS.  We will show that \href{https://oeis.org/A088221}{A088221} counts pairs of connected chord diagrams (allowing empty diagrams).  This gives a combinatorial interpretation for part of the closed form of the asymptotic expansion of $C_n$.

\end{abstract}

\section{Introduction}\label{sec intro}

The concrete result of this paper is a bijection between ordered pairs of chord diagrams and certain rooted trees with a chord diagram structure at each vertex, giving a new combinatorial interpretation for entry \href{https://oeis.org/A088221}{A088221} of the OEIS.  However, this result is more than just an incidental bijection, but rather a preliminary step in a much bigger question about when we can obtain combinatorial understandings of asymptotic expansions and transseries.

We come at both the concrete result and the bigger question from two directions.  On one side we come as pure combinatorialists following the classic path of seeing divergent series as formal power series first, looking to use different kinds of formal expansions for counting, and finding beautiful and insightful bijections between classes of combinatorial objects that explain identities between their counting sequences.  On the other side we come more as mathematical physicists interested in working towards a better understanding of resurgence and instanton expansions in a way which is explicit and applicable in specific cases of interest in quantum field theory.  However, our concrete work herein is purely combinatorial, and skipping this introduction and the concluding discussion, the paper can be read without any other background.

Transseries are a kind of formal expansion allowing many more expressions than powers of $x$ as monomials.  Transseries are important both in analysis and in logic -- a sign of their fundamental importance.  For the application we will investigate it suffices to consider the special case of expansions of the form
\[
    \sum_{n,m\geq 0} c_{n,m}x^{n}(e^{-1/2x}x^{-1/2})^m
\]

An important problem in quantum field theory is to understand how (and to what extent) non-perturbative effects can be recovered from perturbative approaches.  In particular we may wish to start from the Hopf algebraic formulation of perturbative quantum field theory of Connes and Kreimer and recover non-perturbative information. This is one of the reasons one of us studies Dyson-Schwinger equations.  Recently, resurgent analysis and the theory of transseries have been brought to bear on this question \cite{BMV, DunneUnsal, BellonClavier}.  In the simple transseries form above, the $m=0$ part is the perturbative part, $m=1$ is the first instanton part, $m=2$, the second instanton part, and so on.  The instantons are non-perturbative effects.   In \cite{Geraldmichi} Borinsky and Dunne pursue a detailed case study of a particular Dyson-Schwinger equation in Yukawa theory which had been solved perturbatively and given an exact functional solution by Broadhurst and Kreimer \cite{bkerfc}.  From the Dyson-Schwinger equation Borinsky and Dunne are able to give and understand the instanton expansion of this same problem to all orders and so reveal the non-perturbative structure.

However, this story is not just Hopf algebraic and analytic, but also combinatorial.  The simple transseries form can be seen as a bivariate formal series, with the second variable $\xi = e^{-1/2x}x^{-1/2}$ suggestively named.  The perturbative solution to the Dyson-Schwinger equation studied by Broadhurst and Kreimer is essentially the generating function for rooted connected chord diagrams, something which has been explored further in \cite{yu, terminal, con, Karenmarkushihn, michi}.  The first instanton expansion, then, is just an appropriate normalization of the asymptotic expansion of number of rooted connected chord diagrams.  The coefficients of this expansion are, after taking care of some straightforward signs and denominators, also a sequence of positive integers, and through Borinsky's theory of factorially divergent power series \cite{michi}, there is an expression for this expansion in terms of the generating function for rooted connected chord diagrams.  

This situation calls out for a combinatorial interpretation.  Though we are only able to give such an interpretation for part of the resulting function, just this part yields the highly non-trivial new bijection mentioned above.  After presenting the background and details of this bijection, we will conclude the paper by a discussion of the all orders transseries solution of Borinsky and Dunne, and what an as yet unknown combinatorial understanding of it should look like.  Borinsky's theory of factorially divergent power series is also an important step towards a combinatorial interpretation, as it gives general tools for relating a formal series to the series of the asymptotic expansion of its coefficients entirely at the level of formal power series.  However, these series manipulations do not in general maintain good combinatorial interpretations throughout the process, notably because of cancellations between parts with different signs.

\section{Chord Diagrams}\label{chorddiagrams section}

(Connected) chord diagrams stand as a rich structure that becomes handy and informative in a variety of contexts, including bioinformatics \cite{bioinfo}, quantum field theory \cite{karenbook, michi1, michi}, and data structures \cite{datast}, as well as in pure combinatorics \cite{flajoletchords}. 
%Our interest in chord diagrams comes from the context of quantum field theory, in particular, from the Dyson-Schwinger equations \cite{karenbook, Karenmarkushihn,yu}. The solutions to Dyson-Schwinger were recently shown to be described as series indexed by connected chord diagrams with extra conditions on the placing of terminal chords \cite{yu}. Our work on chord diagrams here however will be purely combinatorial. 

%\begin{dfn}[Chord diagrams]
%A \textit{chord diagram} on $n$ chords (i.e. of size $n$) is geometrically perceived simply as a circle with $2n$ nodes that are matched into disjoint pairs, with each pair corresponding to a \textit{chord}. 
%\end{dfn}

\begin{dfn}[Rooted chord diagrams]
%A \textit{rooted} chord diagram is a chord diagram with a selected node. The selected node is called the \textit{root vertex}, and the chord with the root vertex is called the \textit{root chord}. In other words, 
A rooted chord diagram of size $n$ is a matching of the set $\{1,\ldots,2n\}$. Each pair in the matching is a \textit{chord}.
%For an algebraic definition, this is the same as a fixed-point free involution in $S_{2n}$.

We visualize rooted chord diagrams by drawing $\{1,\ldots, 2n\}$ as vertices counterclockwise on a circle, or lined up from left to right, with the chords indicated by lines crossing the circle or arcs above the line.  The vertex $1$ is the \textit{root vertex}, and the chord involving $1$ is the \textit{root chord}.  With the root indicated, the labels $\{1,\ldots, 2n\}$ are unnecessary and so will not be drawn.
\end{dfn}

Then the generating series for rooted chord diagrams is 

\begin{equation}\label{rootedchorddiagsgen}
    D(x):=\sum^\infty_{n=0}(2n-1)!!\; x^n
\end{equation}

All chord diagrams considered here will be rooted and so, when we say a chord diagram we tacitly mean a rooted one.

%Now, a rooted chord diagram can be represented in a linear order, by numbering the nodes in counterclockwise order, starting from the root which receives the label `$1$'. 
The representation of a chord diagram with vertices in a line and chords as arcs above the line will be called the \textit{linear representation} of the chord diagram.
A chord in the diagram may be referred to as $c=\{a<b\}$, where $a$ and $b$ are vertices.

\begin{dfn}[Intervals]
In the linear representation of a rooted chord diagram, an $interval$ is the space to the right of one of the nodes in the linear representation. Thus, a rooted diagram on $n$ chords has $2n$ intervals.
\end{dfn}
 
Note the definition of interval includes the space to the right of the last node in the linear order. 

\begin{figure}[h!]
    \centering
    \includegraphics[scale=0.6]{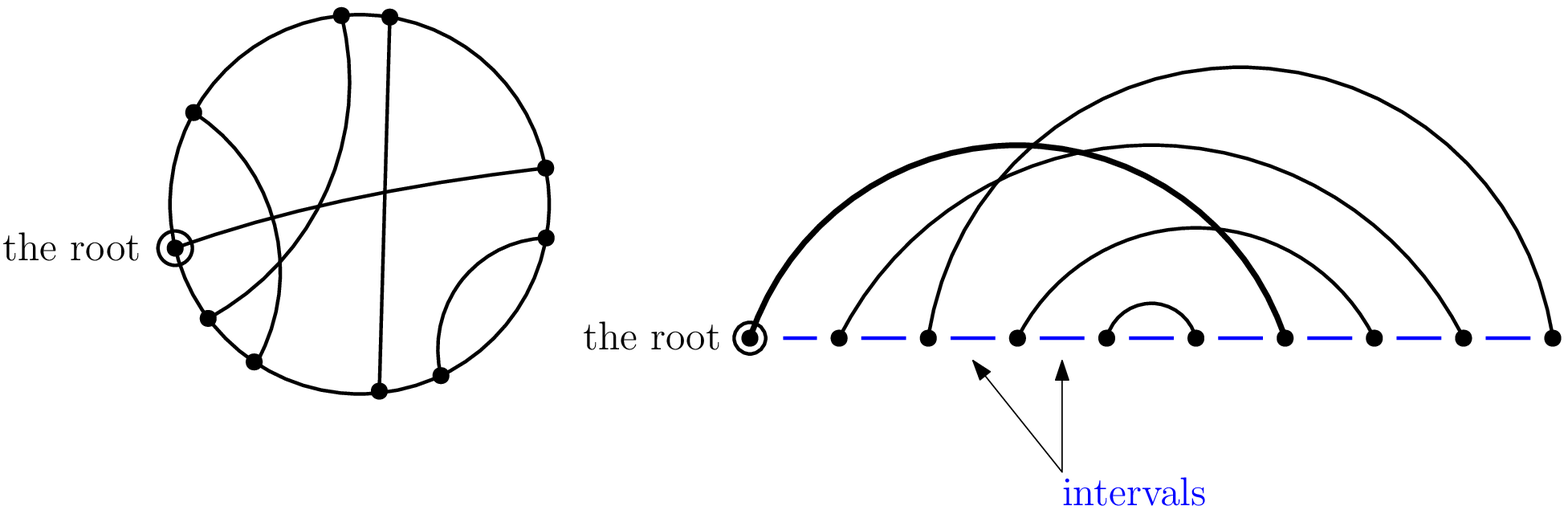}
    \raisebox{1.283cm}{\includegraphics[scale=0.6]{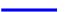}}
    \caption{A rooted chord diagram and its linear representation}
    
\end{figure}

As may be expected by now, the crossings in a chord diagram encode much of the structure and so we ought to give proper notation for them. Namely, in the linear order, two chords $c_1=\{v_1<v_2\}$ and $c_2=\{w_1<w_2\}$ are said to \textit{cross} if $v_1<w_1<v_2<w_2$ or $w_1<v_1<w_2<v_2$.
Keeping track of all the crossings in the diagram leads to the following definition:

\begin{dfn}[The Intersection Graph] 
Given a (rooted) chord diagram $D$ on $n$ chords, consider the following graph $\mathcal{G}_D$: the chords of the diagram will serve as vertices for the new graph, and there is an edge between the two vertices $c_1=\{v_1<v_2\}$ and $c_2=\{w_1<w_2\}$ if $v_1<w_1<v_2<w_2$ or $w_1<v_1<w_2<v_2$, i.e. if the chords cross each other. The graph so constructed is called the \textit{intersection graph} of the given chord diagram.  
\end{dfn}

\begin{rem}
A labelling for the intersection graph can be obtained as follows: give the label $1$ to the root chord;  order the components obtained if the root is removed according to the order of the first vertex of each of them in the linear representation, say the components are $C_1,\ldots,C_n$; and then recursively label each of the components proceeding in that order. It is easily verified that a rooted chord diagram can be uniquely recovered from its labelled intersection graph.\\
\end{rem}

\begin{dfn}[Connected Chord Diagrams]\label{c}
A (rooted) chord diagram is said to be \textit{connected} if its  intersection graph is connected in the graph-theoretic sense. A \textit{connected component} of a diagram is a subset of chords which itself forms a connected chord diagram. The term \textit{root component} will refer to the connected component containing the root chord.
\end{dfn}

\begin{exm}
The diagram $D$ below is a connected chord diagram in linear representation, where the root node is drawn in black.

\begin{center}
\includegraphics[scale=0.5]{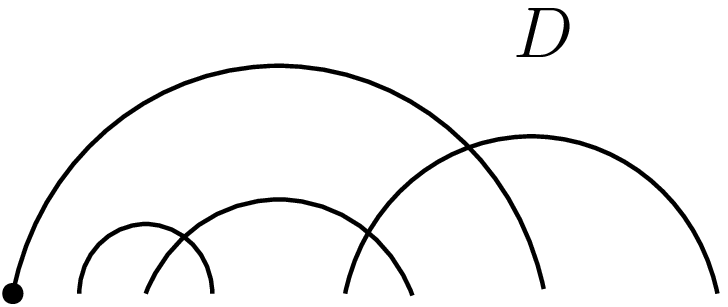}\end{center}
\end{exm}

The generating function for connected chord diagrams (in the number of chords) is denoted by $C(x)$. Thus $C(x)=\sum_{n=0}C_n x^n$, where $C_n$ is the number of connected chord diagrams on $n$ chords. The first terms of $C(x)$ are found to be 
\[C(x)=x+x^2+4x^3+27\;x^4+248\;x^5+\cdots\;;\]
the reader may refer to OEIS sequence \href{https://oeis.org/A000699}{A000699}  for more coefficients. 
The next lemma lists some classic decompositions for chord diagrams (see \cite{flajoletchords} for example), each of which can be used to obtain $C(x)$.

\begin{lem}\label{cd}
  If $D(x), C(x)$ are the generating series for chord diagrams and connected chord diagrams respectively, then 
  \begin{enumerate}
      \item[$\mathrm{(i)}$]  $D(x)=1+C(xD(x)^2)$, \label{i1}
      \item[$\mathrm{(ii)}$] $D(x)=1+xD(x)+2x^2D'(x)$, and \label{i2}
      \item[$\mathrm{(iii)}$] $2xC(x)C'(x)=C(x)(1+C(x))-x$.\label{i3}
  \end{enumerate}

  \end{lem}
  
  \begin{proof} We sketch the underlying decompositions as follows: 
  \begin{enumerate}
    \item[$\mathrm{(i)}$] The `one' term is for the empty chord diagram. Now, given a nonempty chord diagram, we see that for every chord in the root component there live two (potentially empty) chord diagrams to the right of its two ends. This gives the desired decomposition.
     
     \item[$\mathrm{(ii)}$] There are three situations for a root chord: it is either non-existent in the case of the empty diagram, or it is concatenated with a following diagram, or the root chord has its right end landing in one of the intervals of a diagram. These situations correspond respectively with the terms in (ii).
     
   \item[$\mathrm{(iii)}$] Can be derived from (i) and (ii). Nevertheless, it can be also shown as follows: if we remove the root chord what is left is a sequence of connected components, with each component having a special interval (through which the root used to pass) which cannot be the last interval (see the figure below). Thus each of these components is counted according to the generating function $2xC'(x)-C(x)$.
   \begin{center}
   \includegraphics[scale=0.65]{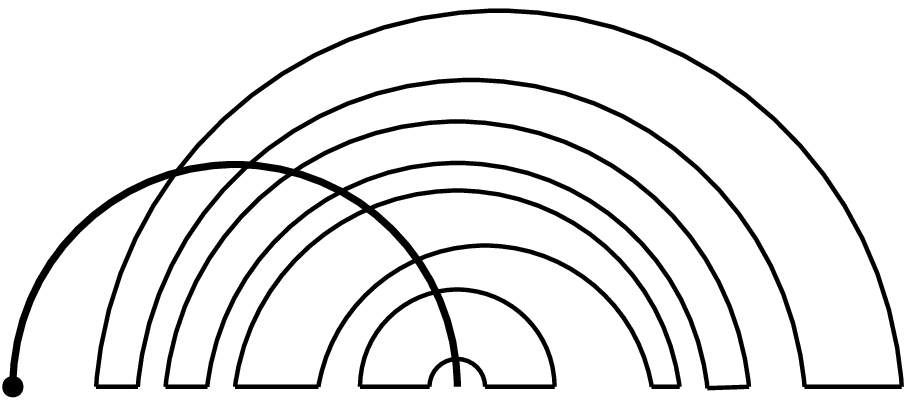}   
   \end{center}
   
   This decomposition gives that
   \[C(x)=\displaystyle\frac{x}{1-(2xC'(x)-C(x))},\]
   and the result follows.
     \end{enumerate}
  \end{proof}

  \setlength{\parindent}{0cm}
  
    We end this section with the definition of an \textit{indecomposable} chord diagram, these diagrams will become a key ingredient later on. 
    \begin{dfn}\label{indecompo}
    A chord diagram is said to be \textit{indecomposable} if, when represented linearly, it is not the concatenation of disjoint nonempty chord diagrams. The empty diagram is vacuously indecomposable by definition. The generating function for indecomposable chord diagrams is denoted here by $I(x)$. We shall also use $I_0(x)$ to denote the generating function for nonempty indecomposable chord diagrams (that is $I(x)=1+I_0(x)$).
    \end{dfn}
    
 \begin{exm}

 Consider the following two diagrams:
 
$D_1=\;$ \raisebox{-0.1cm}{\includegraphics[scale=0.56]{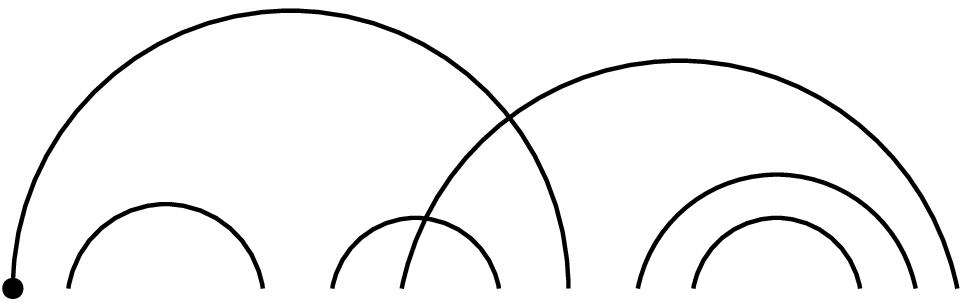}}\;, \;and \;$D_2=\;$ \raisebox{-0.1cm}{\includegraphics[scale=0.54]{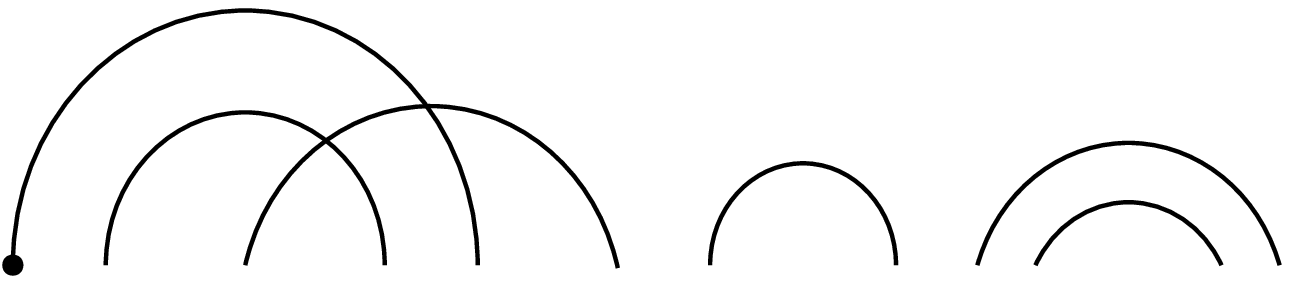}}

\quad Then $D_1$ is indecomposable, whereas $D_2$ is not since it is the concatenation of three indecomposable chord diagrams. Notice that an indecomposable chord diagram is not necessarily connected, $D_1$ provides an example, but the converse is clearly true, namely, any connected diagram is indecomposable. 
 
 \end{exm}
Sequence \href{https://oeis.org/A000698}{A000698} of the OEIS counts indecomposable chord diagrams, the first terms start as 
$$I(x)=1+x+ 2x^2+ 10x^3+74 x^4+ 706x^5+\cdots,$$ where $I(x)$ is the generating series for indecomposable chord diagrams.

\section{Asymptotics of $C_n$}\label{asymptotics o connected}

We pursue a combinatorial interpretation for expressions that appear in the asymptotic expansion of $C_n$, the number of connected chord diagrams on $n$ chords. 
This asymptotic expansion can be expressed as a rational function of $C(x)$ times an exponential function in $C(x)$. One of us, looking at initial terms, conjectured in 2018 that the part in the exponent is following sequence \href{https://oeis.org/A088221}{A088221}. The main result presented in Section \ref{mainhere}  proves this and gives a new combinatorial interpretation for entry  \href{https://oeis.org/A088221}{A088221} of the OEIS.  We will show that \href{https://oeis.org/A088221}{A088221}, surprisingly, counts pairs of connected chord diagrams (allowing empty diagrams). 

\bigskip
%\subsection{Overview}

In \cite{michi1}, M. Borinsky studied the asymptotic behaviour of $C_n$, the number of connected chord diagrams on $n$ chords, as an instance of his work on factorially divergent power series.

First we need a few definitions from from his work.
In \cite{michi, michi1}, M. Borinsky  studied sequences $a_n$ whose asymptotic behaviour for large $n$ follows a relation like
\begin{equation}
     a_n=\alpha^{n+\beta} \Gamma(n+\beta)\bigg(c_0+\displaystyle\frac{c_1}{\alpha(n+\beta-1)}+\displaystyle\frac{c_2}{\alpha^2(n+\beta-1)(n+\beta-2)}+\cdots\bigg),
\end{equation}
where $\alpha\in\mathbb{R}_{>0}$, and $\beta, c_k \in \mathbb{R}$, and 
where $\Gamma(z)=\int_0^\infty x^{z-1}e^{-x}dx$ for $\text{Re}(z)>0$ is the gamma function.

When $a_n$ has such an asymptotic expansion, then the formal power series $f(x) = \sum_{n\geq 0}a_nx^n$ is, following Borinsky, said to be a \emph{factorially divergent power series} and the set of all such is written $\mathbb{R}[[x]]^\alpha_\beta$.  For $\alpha>0$,the $\mathbb{R}[[x]]^\alpha_\beta$ are rings, and the coefficients $c_k$ as given above are well-defined, and so we write $c_k^f$ to emphasize the dependence on $f$ and define

\begin{dfn}[\cite{michi1}]\label{map}
 For $\alpha,\beta\in \mathbb{R}$, with $\alpha>0$, let $ \mathcal{A}_\beta^\alpha:\mathbb{R}[[x]]_\beta^\alpha\rightarrow\mathbb{R}[[x]]$ be the map that has the following action for every $f\in\mathbb{R}[[x]]_\beta^\alpha$
 \[(\mathcal{A}_\beta^\alpha f)(x)=\overset{\infty}{\underset{k=0}{\sum}}c_k^fx^k.\]
\end{dfn}
A map of this type is called an \textit{alien derivative (operator)}  in the context of resurgence theory \cite{resurgence}.

\bigskip

Applying this specifically to the generating series $C(x)$ of rooted connected chord diagrams, Borinsky (\cite{michi}, section 4.6.1) shows that
\begin{align}\label{A}
    \Big(\mathcal{A}_{\frac{1}{2}}^2C\Big)(x)&=\displaystyle
    \frac{1+C(x)-2xC'(x)}{\sqrt{2\pi}}\;e^{-\frac{1}{2x}(2C(x)+C(x)^2)}\\
    &=\frac{x}{\sqrt{2\pi}C(x)}\;e^{-\frac{1}{2x}(2C(x)+C(x)^2)}\;, \tag{$\dagger$}
\end{align}
where the second equality is achieved by appealing to (iii) in Lemma \ref{cd}.
Obtaining such a computable formula for $\mathcal{A}_{\frac{1}{2}}^2C$ means that we have all the coefficients for the asymptotic expansion of $C(x)$. As provided in \cite{michi1}, the first coefficients are 
\begin{equation}\label{calc}
   \Big(\mathcal{A}_{\frac{1}{2}}^2C\Big)(x)=\frac{1}{e\;\sqrt{2\pi}}
   \Bigg(1-\frac{5}{2}x-\frac{43}{8}x^2-\frac{579}{16}x^3-\frac{44477}{128}x^4-\frac{5326191}{1280}x^5\cdots\Bigg).
\end{equation}

Translating back to $C(x)$ itself, for large $n$ this implies
\[C_n= e^{-1}\Big((2n-1)!!-\frac{5}{2}(2n-3)!!-\frac{43}{8}(2n-5)!!-\frac{579}{16}(2n-7)!!
\cdots\Big).\]
This result by M. Borinsky provides a full generalization for the computations in the work of Kleitman  \cite{kleit}, Stein and Everett \cite{steinandeveret} and Bender and Richmond \cite{benderandrichmond}, where only the first term in the expansion has been known. Finally, this also tells us that the probability for a diagram on $n$ chords to be connected is 
$e^{-1}(1-\frac{5}{4n})+\mathcal{O}(1/n^2)$.

\bigskip

Now our goal is to give combinatorial interpretations for as much as possible of the expression
%Namely, if we start with an asymptotic expansion that consists of a sum where terms are scalar multiples of (modified) gamma functions, we can then associate the sequence of scalar coefficients to an ordinary power series instead of the gamma functions and study the algebraic advantage of this process. We discussed this in detail in Section \ref{factorially}. So, roughly speaking, it is shown in \cite{michi1} that, after factoring out the factorial divergencies from the coefficients of the generating function $C(x)$ of   connected chord diagrams, we obtain the following power series in terms of $C(x)$:
\begin{equation}\label{exxx}
    \frac{x}{C(x)} \exp\Big(-\frac{1}{2x}(2C(x)+C(x)^2)\Big).
\end{equation}
By (iii) in Lemma \ref{cd}  we can rewrite the expression inside the exponential as $-1$ times
\begin{equation}\label{eq part}
1+\frac{1}{2x} C(x) (4x\frac{d}{dx}-1)C(x).
\end{equation}
Ignoring the 1 and the 1/2, this can be interpreted as the generating function for rooted chord diagrams with at most two connected components, counted by one less than the number of chords. Indeed, a $2x\displaystyle\frac{d}{dx}$ means distinguishing an interval.  Now, except for the last one, there are two ways of using an interval: we can just place the other $C(x)$ in the interval, or we can place it and pull its root chord to the very front to become the new root. The last interval can only be used in the first way. The coefficients of the expression in \eqref{eq part}, after ignoring the 1 and the 1/2, start as $3,10,63,558,6226,82836,\ldots$, which coincide with those of the sequence \href{https://oeis.org/A088221}{A088221} of the OEIS: $1,2,3,10,63,558,6226,82836,\ldots$. The only definition available for the latter is in terms of another sequence: \href{https://oeis.org/A000698}{A000698} which interestingly counts indecomposable chord diagrams. Namely, the definition tells that (in abuse of notation!) $[x^n](\href{https://oeis.org/A088221}{A088221})^n=[x^{n+1}] I(x)$.
  So, there has to be some bridge between chord diagrams with at most two connected components and indecomposable chord diagrams,  as we shall here prove. 
  
As mentioned earlier, the problem of finding a better combinatorial interpretation for \href{https://oeis.org/A088221}{A088221} lies as a piece in a more general context. We would like to more generally give combinatorial interpretations for  the action described by the map $\mathcal{A}_\beta^\alpha:\mathbb{R}[[x]]_\beta^\alpha\rightarrow\mathbb{R}[[x]]$, as it is applied in the case study of Borinski and Dunne \cite{Geraldmichi} or more generally. %, acting over the subring of $(\alpha,\beta)$-factorially divergent power series (see Capter \ref{chfurther}). 

\bigskip

Throughout the paper we shall stick to the following notation:
\begin{enumerate}
    \item $\mathcal{D}$ is the class of chord diagrams,
    \item $\mathcal{C}$ is the class of connected chord diagrams,
    \item $\mathcal{C}^*$ is the class of connected chord diagrams excluding the one chord diagram,
    \item $\mathcal{D}_{\leq2}$ is the class of chord diagrams with at most two connected components,
    \item $\mathcal{I}$ is the class indecomposable chord diagrams, and finally
    \item $\mathcal{I}_2$  will stand for indecomposable chord diagrams with exactly two components.\\
\end{enumerate}

\section{The Main Bijection}\label{mainhere}
\setlength{\parindent}{0.6cm}

In this section we derive a bijection described by a reversible algorithm to move
 between the class of two lists of indecomposable chord diagrams (allowing empty lists) and the class of rooted trees, in which vertices are of special type, and where a $\mathcal{D}_{\leq2}$-structure is set over the children of every vertex. Our goal by proving this is to prove that \href{https://oeis.org/A088221}{A088221} counts pairs of connected chord diagrams. It may be possible to get the same result algebraically, but we are interested in the combinatorial argument represented mainly in the bijection displayed below.

First we recall the decomposition of a chord diagram by means of extracting the root component (Lemma \ref{cd}-(i))\;:
\[D(x)=1+C(xD(x)^2).\]
This decomposition will be of great help in the construction presented here, hence it may be wise to accompany it with a suitable notation.\\
  
 \begin{nota}\label{notationdangling} The two diagrams that correspond to each chord in the root component will be referred to as the \textit{right dangling} and the \textit{left dangling} diagrams. Given a chord diagram $D$, the root component will be denoted $C_\bullet(D)$, while the dangling diagrams will then be $d_r$ and $d_l$. The symbols $d_r,d_l$ and $C_\bullet$ will often be used as operators.\end{nota}
   
  \begin{exm}\label{exampledangling}
   Consider the following chord diagram.
   \begin{center}
   \raisebox{0cm}{\includegraphics[scale=0.68]{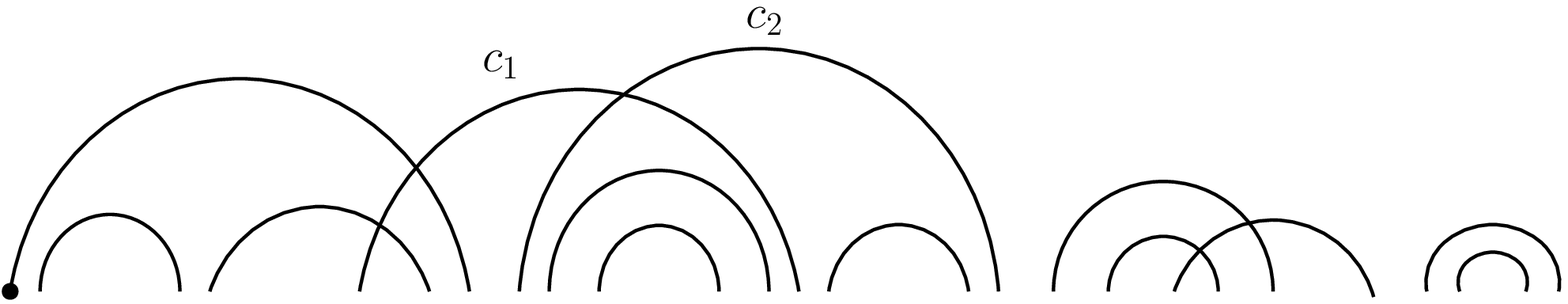}}
   \end{center}

 Let us draw it in a non-standard way
    \begin{center}
   \raisebox{0cm}{\includegraphics[scale=0.68]{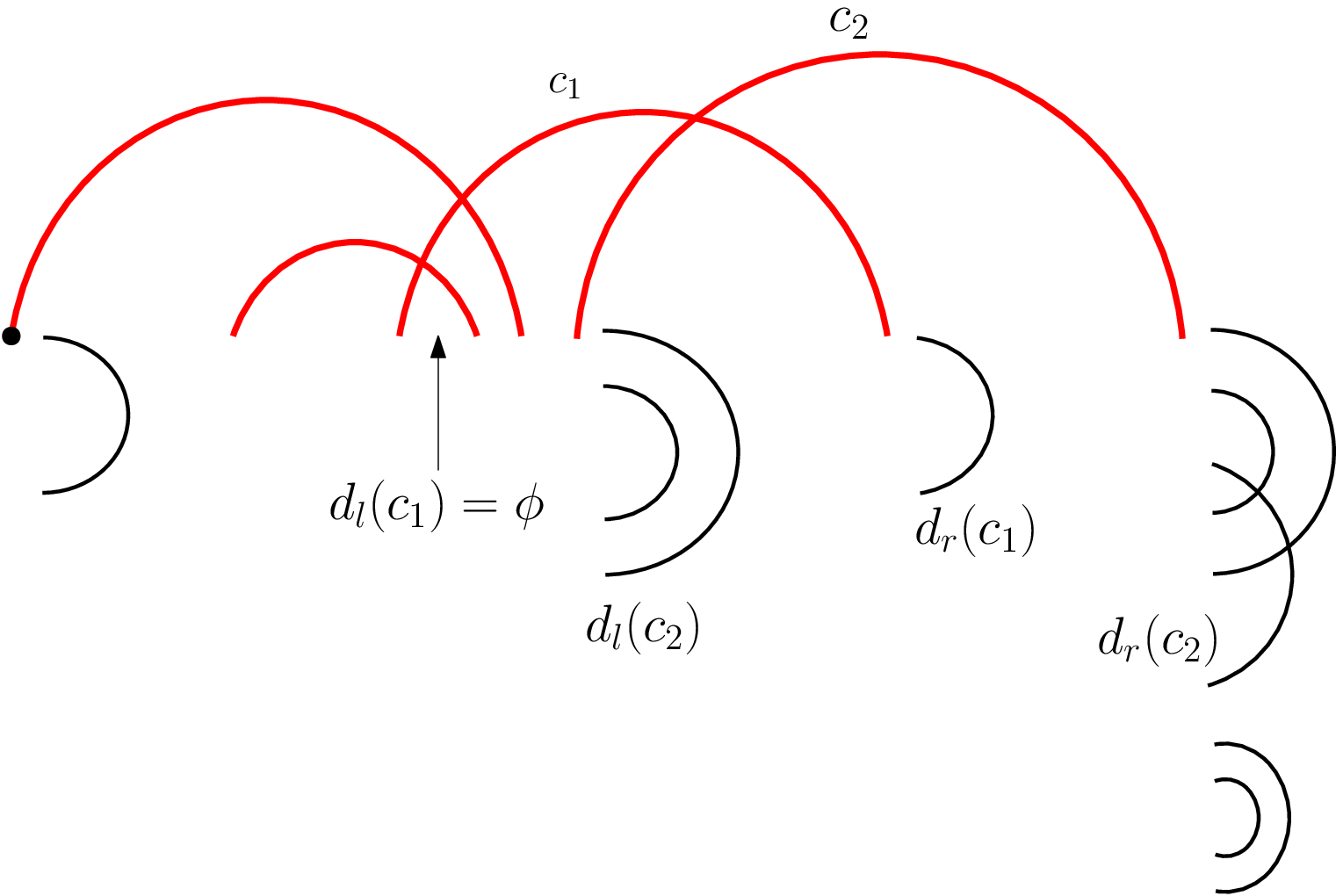}}    
   \end{center}

    which clarifies the decomposition of the lemma. Note that the thick red diagram is the root component $C_\bullet$ of the original diagram. Also, notice that, for example, $d_l(c_1)=\varnothing$. The reason for the nomenclature is now hopefully justified.\\
  \end{exm}

\begin{lem}\label{bij}There is a bijection $\Phi$ between    the class $\mathcal{C}^*$ of rooted connected chord diagrams excluding the one chord diagram, and the class $\mathcal{I}_2$ of indecomposable chord diagrams with exactly two components. Thus, in terms of generating functions
    $\mathcal{I}_2(x)=C(x)-x$.
    \end{lem}
\begin{proof}
The bijection $\Phi$ defined here works almost the same as what is known as the {\it{root share composition}}: Let $C$ be a rooted connected chord diagram. Removing the root chord shall generally leave us with a list of rooted connected components ordered in terms of intersections with the original root. The first  of these components is denoted as $C_2$, while $C_1$ is obtained by removing $C_2$ from the original diagram $C$. Then $(k,C_1,C_2)$ where $1\leq k\leq 2|C_2|-1$, is the {\it{root share decomposition}}  of $C$ (see \cite{yu}). $\Phi$ is a small variant on this.  $\Phi$ places the whole $C_1$ at the interval of $C_2$ where the last end of the original root of $C$ would be if the non-root chords of $C_1$ are removed. Thus, the image $\Phi(C)$ is an indecomposable  chord diagram with exactly two components. This definition is reversible. Indeed, given an indecomposable chord diagram with exactly two connected components, $C_2$ will be the outer component and $C_1$ should be the inner one, and $C$ is obtained by just pulling out the first end of the root of $C_1$ to the leftmost position.
\end{proof}

\begin{exm}
Under the map $\Phi$, the chord diagram
\begin{center}
 \raisebox{-0.1cm}{\includegraphics[scale=0.5]{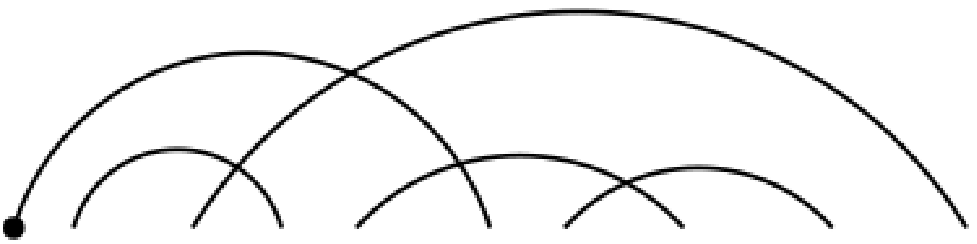}}
 \;\;\text{is mapped to}\;\;
 \raisebox{-0.1cm}{\includegraphics[scale=0.5]{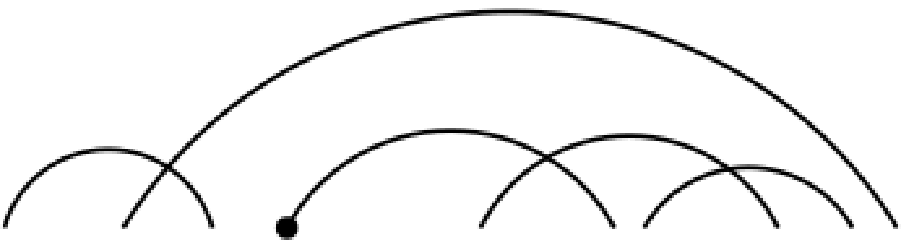}}\;,
\end{center}
\end{exm}

\setlength{\parindent}{0cm}%%%%%%%%%%%%%%%%%%
where, of course, the original root (black) is no longer the root for the resulting diagram.\\

\begin{nota}
In the next theorem, given a finite set $S$ and a class $\mathcal{G}$ of combinatorial objects, the term $\mathcal{G}$-structure on $S$ will mean an arrangement of the elements of $S$ into an object from $\mathcal{G}$. The operation $\ast$ stands for the usual ordered product for combinatorial classes. For example, if $\mathcal{T}$ is the class of trees and $\mathcal{K}_{or}$ is the class of oriented complete graphs, then an element from the class $\mathcal{T}\ast\mathcal{K}_{or}$ will be an ordered pair $(T,K)$ where $T\in\mathcal{T}$ and $K\in\mathcal{K}_{or}$. Notice that, for such a product structure to be applied on a finite set there has to be a partition of the set.

Concretely, we only need to consider the labelled situation.  That is, the underlying atoms of an object $a$ will be labelled by $\{1,\ldots, |a|\}$, and an element of $\mathcal{A}\ast \mathcal{B}$ will be an ordered pair $(a,b)$ but where the labels of $a$ and $b$ together run over $\{1,\ldots, |a|+|b|\}$, and so in particular given labelled objects $a$ and $b$, there are $\binom{|a|+|b|}{|a|}$ ways to combine $a$ and $b$ into an element of $\mathcal{A}\ast \mathcal{B}$.  This is the usual labelled product and is the special case of the more general construction where the sets that the classes are applied on are always $\{1,\ldots, n\}$ with $n$ the size of the object. In this notation, $\mathcal{X}$ stands for the (class of) single object of size 1. 

In this labelled context we work with exponential generating functions, that is if $\mathcal{A}$ is a class of labelled structures then the exponential generating function is $A(x) = \sum_{a\in \mathcal{A}} x^{|a|}/|a|!$.

The reader unfamiliar with this notation from enumerative combinatorics can refer to \cite{goulden}.

\end{nota}

       Note that for any class of chord diagrams the (ordinary) generating function of the class of diagrams and the exponential generating function of the labelled diagrams obtained from all labellings of the elements of the original class is the same since all $n!$ labellings of a chord diagram on $n$ chords are distinguishable.  Because of this, we can move to labelled chord diagrams in the following theorem and algorithm without changing the formal series which we obtain.  The class of rooted trees we consider $\mathcal{Z}$, on the other hand, is labelled and $Z(x)$ is its exponential generating function.

    \begin{thm}\label{mainth}Let $\mathcal{Z}$ be the class of rooted trees where vertices are nonempty ordered sets and where there is a $\mathcal{D}_{\leq2}$-structure over the  children of every vertex. Then there is a bijection $\Theta$ between $\mathcal{Z}$ and the class $\mathcal{X\ast(D\ast D)}$, where $\mathcal{D}$ is the class of chord diagrams. Consequently, if $Z(x)$ is the generating series for $\mathcal{Z}$, then 
    \begin{equation}
        Z=x\big(\displaystyle\frac{1}{1-I_0}\big)^2,
    \end{equation} where $I_0$ is the generating series for nonempty indecomposable chord diagrams.\end{thm}
    
    We will think of the nonempty ordered sets forming the vertices of the trees as little paths inside the vertex, see Figure~\ref{thetaofP}
    
    \begin{proof}
     Begin with a labelled object $P$ from the class $\mathcal{X\ast(D\ast D)}$. In abuse of notation we shall write $d_r(P)$ and $d_l(P)$ for the two diagrams involved. Also \textit{the chord of} $P$ will mean the part of $P$ coming from $\mathcal{X}$ in the decomposition. Then the corresponding $\mathcal{Z}$-tree is obtained through the following algorithm (in all cases the label of a chord becomes the label of the vertex constructed from that chord): \\
    \noindent\rule{\textwidth}{0.4pt}
    \textbf{Algorithm 1\label{alg}: Make $\mathbf{\mathcal{Z}}$-Tree}\\
     \noindent\rule{\textwidth}{0.4pt}
   
   \texttt{Input:} $P\;$=\;\raisebox{0cm}{\includegraphics[scale=0.3]{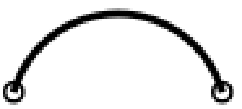}}
   $(d_l,d_r) \;\;$ 
   
  \texttt{initially} $Q_1=P$; \setlength{\parindent}{1.3cm} 
  
  \texttt{queue} $Q=(Q_1)$;
  
  \texttt{integer} $L=\text{length}(Q)$   (automatically modified by any alteration of $Q$);
  
  \texttt{vertex} $v=\odot$\;;
  
  \texttt{label}($v$)\;=\;label given to the chord of $Q_1$; 
  
  \texttt{diagrams} $D_l=D_r=\varnothing$;
  
  \texttt{tree} $Z=v$;\\

   \setlength{\parindent}{0cm} 
 Set $v$ as the root vertex of $Z$\\
  While $Q\neq\varnothing$ \{\\\setlength{\parindent}{1.3cm} 
  \begin{enumerate}
      \item Set $D_l=d_l(Q_1)$ and $D_r=d_r(Q_1)$;\\
  
      \item If $D_l=\varnothing= D_r$ then:

      - push $Q_1$ out of $Q$ (i.e. for all $1\leq k<L, Q_k\leftarrow Q_{k+1}$
      );    
      
      - Go to step (1) again.\\
      
      \item If $D_l= \varnothing$ and $D_r\neq\varnothing$  then:
      
      - Create $|C_\bullet(D_r)|$ children attached to $v$,
      and set their labels to be the same
      
      as the chords in $C_\bullet(D_r)$. Namely, let $\{w_1,\ldots,w_{|C_\bullet(D_r)|}\}$ be the children
      
      and set
      label($w_i$)=label($i^{\text{th}}$ \text{chord}) in the obvious meaning;
      
      - For each $i\in\{1,\ldots,|C_\bullet(D_r)|\}$ add $Q_{L+i}$ to the queue $Q$, where
      
      $Q_{L+i}$:=\;
  \raisebox{0cm}{\includegraphics[scale=0.3]{Figures/onechord.eps}} $\big(d_l(i^{\text{th}} \text{chord}),d_r(i^{\text{th}} \text{chord})\big)$, where the single chord 
   
   is standing for the $i^\text{th}$ chord in $C_\bullet(D_r)$ and the $d_l,d_r$ are the dangling
   
   diagrams of this chord in $D_r$;
   
   - Set $C_\bullet(D_r)$ as the $\mathcal{D}_{\leq2}$-structure over the children of $v$;
   
   - Push $Q_1$ out of $Q$; 
   
   - Set $v=$ vertex for the chord of $Q_1$ (where $Q_1$ has been updated);
   
   - Go to step (1);\\

   \item If $D_l\neq \varnothing$ and $D_r\neq\varnothing$  then:
   
  - Create $|C_\bullet(D_l)|+|C_\bullet(D_r)|$ children attached to $v$,
      and set their labels to
      
      be the same
      as the corresponding chords, as before;
      
      - For each $i\in\{1,\ldots,|C_\bullet(D_l)|+|C_\bullet(D_r)|\}$  add $Q_{L+i}$ to the queue $Q$, where
      
      $Q_{L+i}$:=\;
  \raisebox{0cm}{\includegraphics[scale=0.3]{Figures/onechord.eps}} $\big(d_l(i^{\text{th}} \text{chord}),d_r(i^{\text{th}} \text{chord})\big)$, where the single chord 
   
   is standing for the $i^\text{th}$ chord in $C_\bullet(D_l)$ if $1\leq i\leq|C_\bullet(D_l)|$ and for the
   
   $(i-|C_\bullet(D_l)|)^\text{th}$ chord in $C_\bullet(D_r)$ otherwise;
   
   - Set the concatenation $C_\bullet(D_l) C_\bullet(D_r)$ as the $\mathcal{D}_{\leq2}$-structure over the children 
   
   of $v$;
   
   - Push $Q_1$ out of $Q$; 
   
   - Set $v=$ vertex for the chord of $Q_1$;
   
   - Go to step (1);\\
  %%%%%%%%%%%%%%%%%%%%%%%%%%%%%%%%%%%%%%%%%%%%%%%%%%%%%%%%%%% 
   \item If $D_l\neq \varnothing$ and $D_r=\varnothing$  then:
   \begin{enumerate}
       \item[(a)] In case $C_\bullet(D_l)$ is a \textbf{single chord} $c$ then:
       
       - vertex $v$ \textbf{absorbs} another node with the label given to $c$. It is appropriate to think of a vertex here as some sort of stack comprising labelled nodes:\\
       
       \begin{center}
           \includegraphics[scale=0.8]{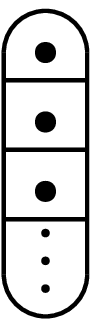}
      \end{center}
       - Add $Q_{L+1}$ to $Q$, where $Q_{L+1}$ consists of $c$ and its dangling diagrams, i.e. $d_l(Q_{L+1})=d_l(c)$ and $d_r(Q_{L+1})=d_r(c)$;  
       
       - Push $Q_1$ out of $Q$; 
   
       - Set $v=$ vertex for the chord of $Q_1$;
   
       - Go to step (1);\\
       
       \item[(b)] Otherwise if $C_\bullet(D_l)$ is \textbf{not a single chord} then:
       
       - Create $|C_\bullet(D_l)|$ children attached to $v$,
       and set their labels to
       be the same
      as the corresponding chords, as before;
      
      - For each $i\in\{1,\ldots,|C_\bullet(D_l)|\}$  add $Q_{L+i}$ to the queue $Q$, where
      
      $Q_{L+i}$:=\;
  \raisebox{0cm}{\includegraphics[scale=0.3]{Figures/onechord.eps}} $\big(d_l(i^{\text{th}} \text{chord}),d_r(i^{\text{th}} \text{chord})\big)$, where the single chord 
   is standing for the $i^\text{th}$ chord in $C_\bullet(D_l)$;

   - Set  $\Phi(C_\bullet(D_l))$ as the $\mathcal{D}_{\leq2}$-structure over the children 
   of $v$;
   
   - Push $Q_1$ out of $Q$; 
   
   - Set $v=$ vertex for the chord of $Q_1$;
   
   - Go to step (1);\;\;\; \}\\
    \end{enumerate}

  \end{enumerate}
   
 \texttt{Output:} $\Theta(P)=Z$.\\
 \noindent\rule{\textwidth}{0.4pt}\\
 
\setlength{\parindent}{0cm}
This algorithm uniquely generates the corresponding tree. Indeed, to see this it shall be enough to see that every branching from a vertex is uniquely translated into chord diagrams:

\begin{enumerate}
    \item If the $\mathcal{D}_{\leq2}$-structure over the children is the concatenation of two connected components, then we know simply that there were nonempty right and left dangling diagrams for the chord corresponding to the vertex. Further, the two connected components are, respectively, the root components of  the dangling diagrams. The order of components in the (rooted) $\mathcal{D}_{\leq2}$-structure dictates which component is for the left or right dangling diagram. 
    
    \item If the $\mathcal{D}_{\leq2}$-structure is just a connected chord diagram, then for the chord corresponding to the vertex only the right dangling diagram was nonempty. In particular, this connected structure is the root component for the right dangling diagram.
    %
    %The next two cases are actually the most crucial.
    \item If the $\mathcal{D}_{\leq2}$-structure is an indecomposable chord diagram with exactly two connected components, then we learn that for the corresponding chord only the left dangling diagram was nonempty. The root component of which is determined by applying $\Phi^{-1}$ to the $\mathcal{D}_{\leq2}$-structure. This process is well-defined by virtue of $\Phi$ being a bijection.
    
    \item The only remaining case is when the vertex itself is a stack. This marks that, as in the previous case, only the left dangling diagram was nonempty for the chord corresponding to the vertex, and that, further, the root component for this dangling diagram was a single chord with the label given next in the stack. Uniqueness in this case is clear as the information is encoded into the tree in a way that does not interfere with the previous cases, hence no ambiguity arises.\\
\end{enumerate}

This outlines that the above algorithm is reversible, and hence establishes the desired bijection. To prove the second part of the theorem notice that any rooted chord diagram can be viewed as a (possibly empty) list of nonempty indecomposable chord diagrams. 
    \end{proof}
    
    \begin{exm}\label{exampleofZbijection}
    Let $P\in\mathcal{X\ast(D\ast D)}$ be given by $P=\;$ \raisebox{0cm}{\includegraphics[scale=0.3]{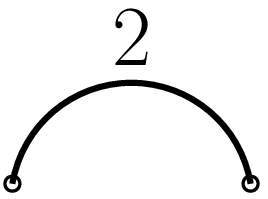}} $\big(d_l(P),d_r(P)\big)$, where

 $d_r(P)=\;$ \raisebox{-0.12cm}{\includegraphics[scale=0.7]{
 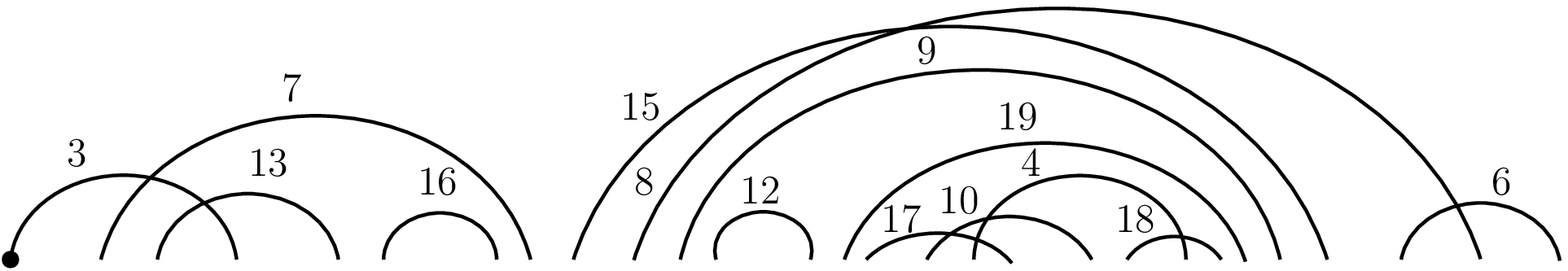}}

  and
    
  $d_l(P)=\;$ \raisebox{-0.12cm}{\includegraphics[scale=0.7]{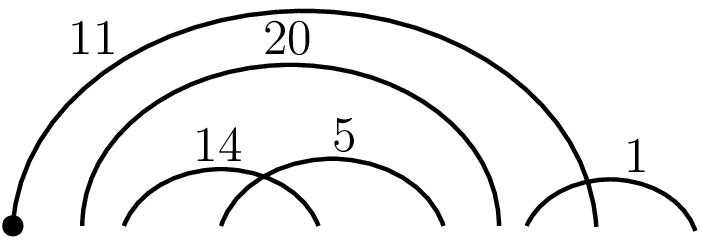}}\;.\\ 

The first iterations in the algorithm are going to be as follows:\\

[1] Initially:

$Q_1=P$,

$Q=(Q_1)$,

$L=1$,

$Z=$ \raisebox{-0.12cm}{\includegraphics[scale=0.7]{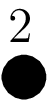}}\;.
 
\noindent\rule{\textwidth}{0.4pt}\\

[2]
$D_l\neq\varnothing$ and $D_r\neq\varnothing$, so we attach the children as of

$Z=\;$ \raisebox{-0cm}{\includegraphics[scale=0.65]{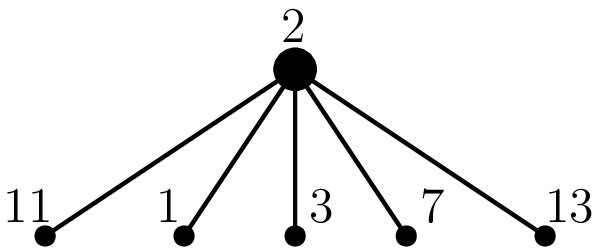}}
   \;with the $\mathcal{D}_{\leq2}$-structure
      \raisebox{-0cm}{\includegraphics[scale=0.6687]{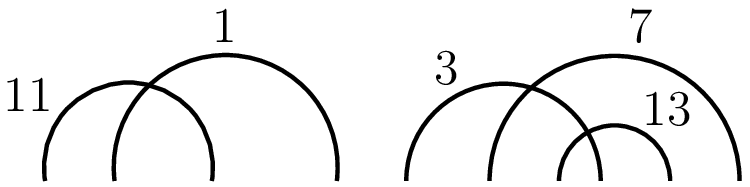}}.
 The updated queue becomes (after including the new entries and pushing the old $Q_1$ out of $Q$):

\[ Q=\bigg(Q_1= \raisebox{-2.6cm}{\includegraphics[scale=0.75]{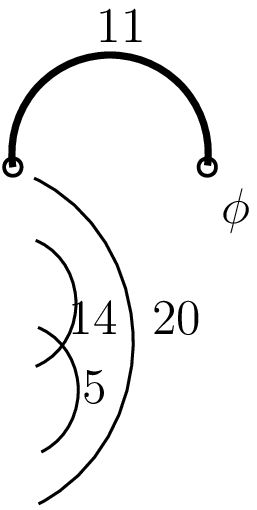}}\;,\;
 \raisebox{-0.5cm}{\includegraphics[scale=0.75]{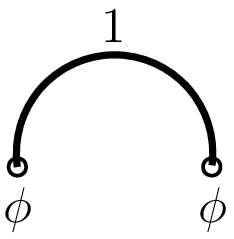}}\;,\;
 \raisebox{-0.5cm}{\includegraphics[scale=0.75]{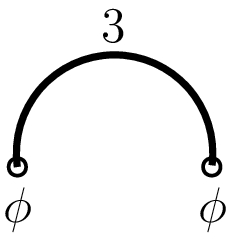}}\;,\;
 \raisebox{-0.5cm}{\includegraphics[scale=0.75]{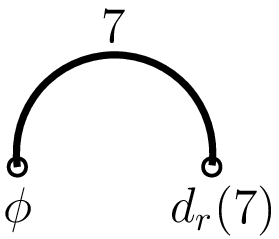}}\;,\;
 \raisebox{-1.65cm}{\includegraphics[scale=0.75]{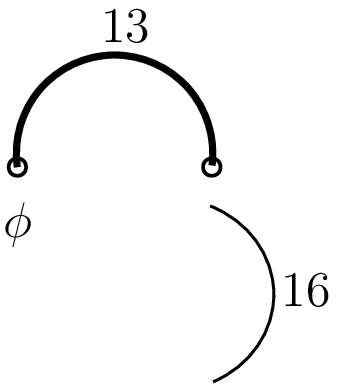}} \bigg)\]

and then the vertex $v$ is set to be $11$.

\noindent\rule{\textwidth}{0.4pt}\\

[3] In this iteration we find that $D_l\neq\varnothing$, whereas $D_r=\varnothing$, moreover, $C_{\bullet}(D_l)$ is the single chord labelled $`20$'. Thus, following the algorithm, one more node is appended to the vertex $v$ which, before this moment, only contained the node labelled $11$. Thus, vertex $v$ is now given by 
\raisebox{-0.12cm}{\includegraphics[scale=0.5]{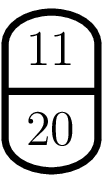}}. Then we add the entry 
$Q_6=\;\raisebox{-0.1cm}{\includegraphics[scale=0.42]{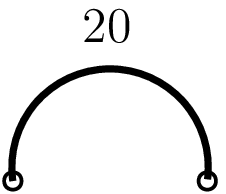}}\;\big(\raisebox{-0.1cm}{\includegraphics[scale=0.7]{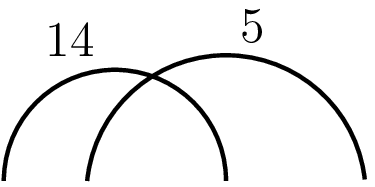}}\;,\varnothing\big)$
to the queue $Q$ (from the end); update $Q$ by pushing out $Q_1$; and set the vertex $v$ to be at chord `1', since it is the chord of the new $Q_1$.\\

%%%%%%%%%%%%%%%%%%%%%%%%%%%%%%%%%%%%%%%%%%%%%%%%%%%%%%%%%%%%%%%%%%%%%%%%%%%%%%%%%%%%%%%%%%%%%%%%%%%%%%%%%%%%%%%%%%%%%%%%%%%%%%%%%%%%%%%%%%%%%%%%

Following the algorithm to the end we generate the tree $\Theta(P)$ to be as in Figure \ref{thetaofP} below, where the right column displays the $\mathcal{D}_{\leq2}$-structures pertinent to the children of each vertex (recall that vertices here are generally stacks of nodes). For clarity, structures are displayed level-wise.

 \begin{figure}[h]
     \centering
    \raisebox{0cm}{\includegraphics[scale=0.7]{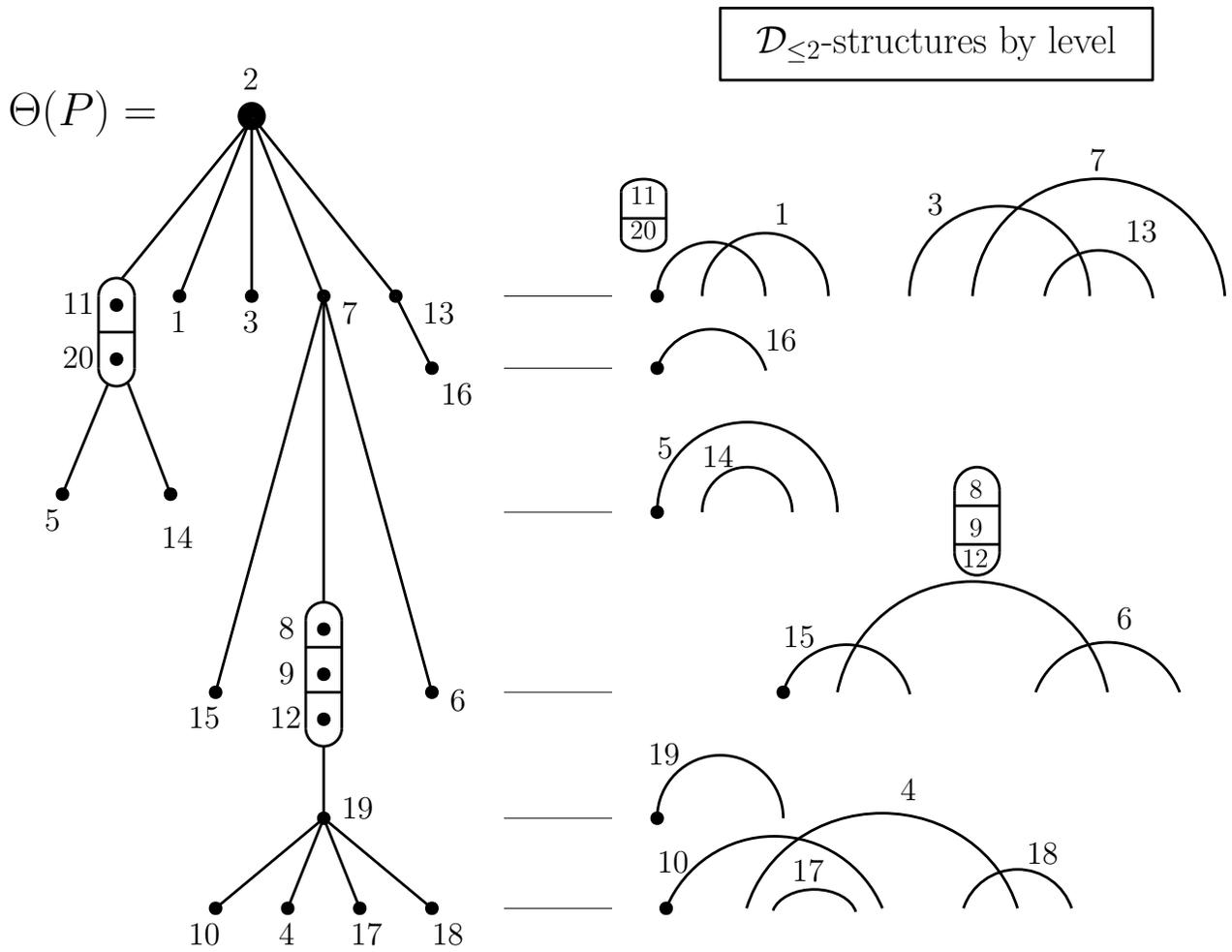}}
     \caption{$\Theta(P)$, with $\mathcal{D}_{\leq2}$-structures displayed by level on the right.}
     \label{thetaofP}
 \end{figure}

 \end{exm} 
 
 \begin{cor}\label{coro} Let $I_0(x)$ be the generating series for nonempty indecomposable chord diagrams as before, and set $B(x)=D_{\leq2}(x)+x$, where $D_{\leq2}(x)$ is the generating function for the class $\mathcal{D}_{\leq2}$. Then $$I_0(x)=\displaystyle\frac{x}{1-xB'(Z)}\;,$$ where $Z$ is the generating series for the class $\mathcal{Z}$ as before.
 
    \end{cor}
 Before proving Corollary \ref{coro} we will prove a decomposition of indecomposable diagrams:
 
 \begin{lem}\label{inde}
   The generating series $I_0$ for nonempty indecomposable chord diagrams satisfies the relation \[I_0(x)=x+\displaystyle\frac{2x^2I_0'(x)}{1-I_0(x)}\;.\label{ind rel}\]
 \end{lem}
 \begin{proof}
     Given a nonempty indecomposable chord diagram we can argue as follows. If the diagram is not a single chord, then removing the root chord generally leaves us with a list of nonempty indecomposable chord diagrams. Moreover, the last diagram in this list carries all the information about the removed root chord, encoded as a marked interval that used to carry the right end of the root chord. Recall that the intervals are the spaces to the right of every chord end in the linear representation, including the last space to the right of the diagram.  Thus we have $2m$ intervals in a diagram with $m$ chords. The relation in the lemma is exactly the translation of this decomposition into the world of generating series. 
 \end{proof}

 \begin{exm}
 In the following diagram, the diagram is decomposed into: the root, $D_1$, $D_2$, and ($D_3$, interval 4), where, among the 8 intervals in $D_3$, the root originally landed in interval 4 (marked by a red dotted line).  
 
\includegraphics[scale=0.85]{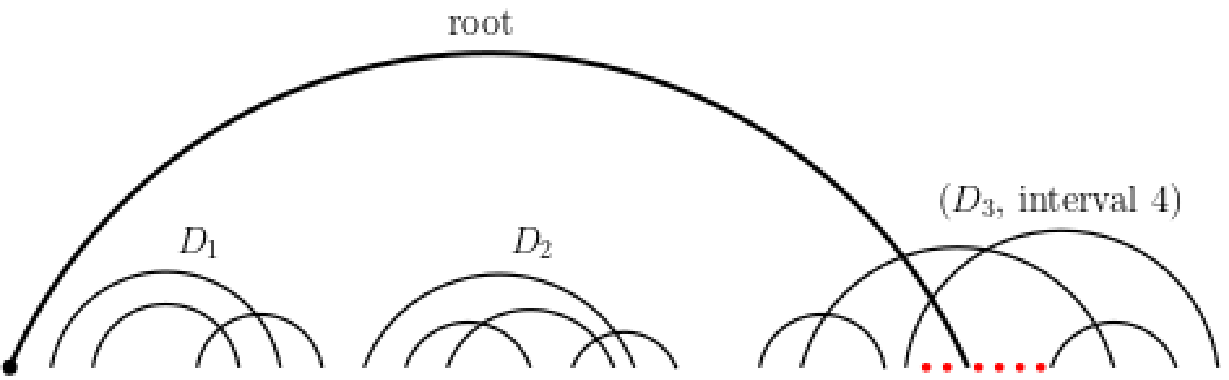}\\
  
 \end{exm}

    \begin{proof}[Proof of Corollary \ref{coro}:]
    First of all, notice that by the definition of the class $\mathcal{Z}$, the generating series $Z$ satisfies the recursion \[Z(x)=xD_{\leq2}(Z)+xZ=x B(Z),\] where $B(t)=D_{\leq2}(t)+t$. Thus $Z'(x)=B(Z)+xB'(Z)Z'=\displaystyle\frac{Z}{x}+xB'(Z)Z',$ and hence \[Z=xZ'(1-xB'(Z)).\]By Theorem \ref{mainth}, we have that $Z=x\Big(\displaystyle\frac{1}{1-I_0}\Big)^2$. Taking the logarithmic derivative of both sides and making use of the above identities we get
    \begin{align*}
        1+2x\frac{d}{dx}\log\Big(\displaystyle\frac{1}{1-I_0}\Big)&=x\displaystyle\frac{d}{dx}\log Z= x\displaystyle\frac{Z'}{Z}
        =\displaystyle\frac{1}{1-xB'(Z)},
    \end{align*}
    
and hence $$\frac{1}{1-xB'(Z)}=1+\displaystyle\frac{2xI_0'}{(1-I_0)}\;.$$

Multiplying by $x$ we get that, by Lemma \ref{ind rel}, $\displaystyle\frac{x}{1-xB'(Z)}=x+\displaystyle\frac{2x^2I_0'}{(1-I_0)}=I_0\;,$  which completes the proof.
    \end{proof}

    \begin{cor}\label{endcoro} Let $A(X)$ be the generating series for the sequence \href{https://oeis.org/A088221}{A088221}. Then $$A(x)=D_{\leq2}(x)+x.$$\end{cor}
\begin{proof}
By Lagrange inversion\footnote{Lagrange inversion, also known as the Lagrange implicit function theorem, can be found in many standard enumeration references.  See for example \cite{goulden}.  Briefly, in the form that we will use it, Lagrange inversion says that if $Z(x) = xB(Z(x))$ then $[x^n]Z(x) = \frac{1}{n}[x^{n-1}]B(x)^n$.}, we know that 
\[[x^n]\displaystyle\frac{1}{1-xB'(Z)}=[x^n]B^n(x).\]
 Now, by the definition of the sequence \href{https://oeis.org/A088221}{A088221}, we know that it is the sequence for which $[x^n]A^n(x)=[x^{n+1}]I_0(x)$, where $A(x)$ is assumed to be the generating series for the sequence \href{https://oeis.org/A088221}{A088221}. This gives that $[x^n]A^n(x)=[x^{n+1}]I_0(x)=
 [x^{n+1}]\displaystyle\frac{x}{1-xB'(Z)}=[x^n]\displaystyle\frac{1}{1-xB'(Z)}=[x^n]B^n(x)$, and so $A(x)=B(x)=D_{\leq2}(x)+x.$
\end{proof}

\begin{prop}
The $n^\text{th}$entry in the sequence \href{https://oeis.org/A088221}{A088221} counts the number of pairs $(C_1,C_2)$ of connected chord diagrams (allowing empty diagrams) with total number of chords being $n$.\end{prop}

\begin{proof}
Indeed, any chord diagram with at most two connected components is either: 
(1) empty, (2) connected, (3) concatenation of two connected diagrams, (4) or is indecomposable with exactly two connected components.
    By using Lemma \ref{bij} for the last case we thus get \[A(x)=D_{\leq2}(x)+x=1+C(x)+C^2(x)+C(x)-x+x=(C(x)+1)^2,\] and the result is established.
\end{proof}

\section{Discussion}\label{sec discussion}

The bijections above are interesting in their own right, but for us, they are most interesting as a small first step towards giving a combinatorial interpretation of the analysis of Borinsky and Dunne \cite{Geraldmichi}, and ultimately more generally of such resurgence setups.  To this end, let us further discuss the situation of Borinsky and Dunne.

They work with the transseries Ansatz (equation 20 of \cite{Geraldmichi})
\[
    C(x) = \sum_{k=0}^\infty \sigma^k C^{(k)}(x)
\]
where $\sigma$ is an instanton parameter (as it will turn out it will come along with the same power of $e^{-1/2x}/\sqrt{x}$, giving a transseries in the simplified form mentioned in the introduction.  In the expansion of Borinsky and Dunne, $C^{(0)}(x)$ is what we have called $C(x)$ in this paper, the generating series for connected chord diagrams.  $C^{(0)}(x)$ is also a scaled version of the perturbative solution to the Dyson-Schwinger equation in Yukawa theory that Broadhurst and Kreimer solved in \cite{bkerfc}.  The $C(x)$ of Borinsky and Dunne is required to satisfy the same differential equation as the generating series for the connected chord diagrams, namely Lemma~\ref{cd} (iii).  Taking the coefficient of $\sigma^0$ we see that $C^{(0)}(x)$ must satisfy the same differential equation and so must be exactly the generating series for connected chord diagrams.

They next derive, using the differential equation, the expression for $C^{(1)}$ whose exponential part we have been studying in this paper.  
\begin{align*}
    C^{(1)}(x) & = \left(\frac{e^{-1/2x}}{\sqrt{x}}\right)\left(\mathcal{A}^2_{\frac{1}{2}}C^{(0)}\right)(x) \\
    & = 
    \left(\frac{e^{-1/2x}}{\sqrt{x}}\right) \left(\frac{x}{\sqrt{2\pi}C^{(0)}(x)}\right)e^{-\frac{(C^{(0)}(x)(C^{(0)}(x)+2)}{2x}} \\
    & = \xi\frac{1}{e\;\sqrt{2\pi}}
   \Bigg(1-\frac{5}{2}x-\frac{43}{8}x^2-\frac{579}{16}x^3-\frac{44477}{128}x^4-\frac{5326191}{1280}x^5\cdots\Bigg),
\end{align*}
where $\xi = e^{-1/2x}/\sqrt{x}$.

In fact, their methods make it possible to give expressions for all the $C^{(j)}(x)$ in terms of $C^{(0)}(x)$ and derivatives of an explicit bivariate function they call $f(x,y)$ (see equation 38 of \cite{Geraldmichi}).  Expanding the first few of these expressions one obtains
\begin{align*}
    C^{(2)}(x) & = \xi^2\frac{1}{2\pi e^2} \left(-\frac{1}{x} + 5 + \frac{11}{2}x + \frac{97}{2}x^2 + \frac{4173}{8}x^3 + \frac{268051}{40}x^4 + \cdots  \right) \\
    C^{(3)}(x) & = \xi^3\frac{1}{(2\pi)^{3/2} e^3}\left(\frac{3}{2x^2} -\frac{47}{4x} + \frac{67}{16} -\frac{2157}{32}x -\frac{211199}{256}x^2 -\frac{29245909}{2560}x^3 + \cdots \right)
\end{align*}

%***Ali, do you know why the sign I get is different from their equation 36? -- my sign works correctly in the square***

Borinsky and Dunne also consider expanding $C(x) = \sum_{k=0}^\infty \sigma^k C^{(k)}(x)$ in terms of $x$ with the coefficients series in $\sigma\xi$.  They write (equation 39 of \cite{Geraldmichi})
\[
    C(x) = x\sum_{n=0}^\infty x^n F_n(\rho)
\]
where $\rho = \sigma\xi/x$.  Then they find that the same differential equation gives that $F_0(\rho) = 1 + W(\rho)$ where $W$ is the Lambert $W$ function.  The Lambert $W$ function is defined by the identity $W(\rho)e^{W(\rho)} = \rho$.  From an enumerative perspective it is a standard fact that the Lambert $W$ function gives the exponential generating series for nonempty labelled rooted trees.  Specifically letting $R(z)$ be the exponential generating series of labelled rooted trees we have $R(z) = ze^{R(z)}$ since a nonempty rooted tree is a root and a forest of subtrees.  Rearranging we have $R(z)e^{-R(z)} = z$ and so $R(z) = -W(-z)$.  This means that $F_0$ is the exponential generating series of labelled rooted trees, allowing an empty tree, and with a sign of $(-1)^{|t|-1}$ for a tree $t$ with $|t|$ vertices.  That is,  $F_0(\rho) = 1 - R(-\rho)$.

Borinsky and Dunne were well aware of this connection to rooted trees \cite{BDpersonal}.  It makes the question of an overall combinatorial interpretation all the more tantalizing, for what we have is a square
\[
\begin{array}{c|ccccc}
 & F_0(\rho) & F_1(\rho) & F_2(\rho) & F_3(\rho) & \cdots \\
 \hline
 C^{(0)}(x)\,\frac{1}{x} & 1 & 1 & 4 & 27 & \cdots \\
 C^{(1)}(x)\,\frac{e \sqrt{2\pi}}{\xi} & 1 & -\frac{5}{2} &  -\frac{43}{8}& -\frac{579}{16} & \cdots \\
 C^{(2)}(x)\,\frac{xe^2 2\pi}{\xi^2} & -1 & 5& \frac{11}{2}&  \frac{97}{2} & \cdots \\
 C^{(3)}(x)\,\frac{x^2e^3(2\pi)^{3/2}}{\xi^2} & \frac{3}{2} & -\frac{47}{4} & \frac{67}{16}& -\frac{2157}{32} & \cdots \\
 \vdots & \vdots & \vdots & \vdots & \vdots & \ddots
\end{array}
\]
where the first row and first column both have classical combinatorial interpretations: the ordinary generating series of rooted connected chord diagrams and the exponential generating series of labelled rooted trees, respectively.  For the second row, we have taken a small step, giving a combinatorial interpretation for part of what goes into that series.  

We would like to understand this whole square combinatorially, namely to find some combinatorial objects counted with respect to two parameters, where setting one parameter to $0$ gives rooted trees and setting the other parameter to $0$ gives connected chord diagrams, and a bivariate count gives the square above.  One reason to doubt the possibility of this is the signs, which do not follow an obvious pattern.  The principal difficulty we encountered also comes from signs, though in this case from internal signs.  Borinsky and Dunne found equations for all these series in terms of $C^{(0)}$ and the Lambert $W$.  If these equations involved only positive coefficients then by standard methods all the operations would immediately be interpretable combinatorially and the problem would be trivial.  However, the expressions do involve signs, and so giving an interpretation becomes a challenging task of giving sufficiently nice interpretations of the pieces that cancellations coming from sign differences can be understood as explicit set subtractions. In general there is no reason to expect this to be possible, but we remain hopeful that in structured circumstances like this one, with sufficient cleverness, a fully combinatorial explanation will be possible.

\bibliographystyle{plain}
\bibliography{references.bib}

\end{document}